\documentclass[12pt]{amsart}
\usepackage{amsmath,amssymb,amsbsy,amsfonts,latexsym,amsopn,amstext,cite,
                                               amsxtra,euscript,amscd,bm,mathabx,mathrsfs, abraces}
\usepackage{url}
\usepackage[colorlinks,linkcolor=blue,anchorcolor=blue,citecolor=blue,backref=page]{hyperref}
\usepackage{color}
\usepackage{graphics,epsfig}
\usepackage{graphicx}
\usepackage{float} 
\usepackage[english]{babel}
\usepackage{mathtools}
\usepackage{todonotes}
\usepackage{url}
\usepackage[colorlinks,linkcolor=blue,anchorcolor=blue,citecolor=blue,backref=page]{hyperref}

\usepackage[norefs,nocites]{refcheck}

\hypersetup{breaklinks=true}

\usepackage[norefs,nocites]{refcheck}
\usepackage[english]{babel}

\newtheorem{thm}{Theorem}
\newtheorem{lem}[thm]{Lemma}

\newtheorem{cor}[thm]{Corollary}

\newtheorem{defn}[thm]{Definition}

\newtheorem{lemma}[thm]{Lemma}

\newcommand{\SL}{\operatorname{SL}}

\newcommand{\Tr}{\operatorname{Tr}}

\numberwithin{equation}{section}
\numberwithin{thm}{section}
\numberwithin{table}{section}

\def\squareforqed{\hbox{\rlap{$\sqcap$}$\sqcup$}}
\def\qed{\ifmmode\squareforqed\else{\unskip\nobreak\hfil
\penalty50\hskip1em\null\nobreak\hfil\squareforqed
\parfillskip=0pt\finalhyphendemerits=0\endgraf}\fi}

\def\Bsf{\mathsf B}
\def\Esf{\mathsf E}

\def\cB{{\mathcal B}}

\def\cD{{\mathcal D}}

\def \Z {{\mathbb Z}}

 \def\\{\cr}
\def\({\left(}
\def\){\right)}

\def\mand{\qquad\mbox{and}\qquad}
\usepackage{todonotes}

\usepackage[norefs,nocites]{refcheck}

\hypersetup{breaklinks=true}

\newcommand{\bC}{\mathbb{C}}

\newcommand{\bH}{\mathbb{H}}

\newcommand{\bN}{\mathbb{N}}

\newcommand{\bR}{\mathbb{R}}

\newcommand{\bZ}{\mathbb{Z}}

\def\squareforqed{\hbox{\rlap{$\sqcap$}$\sqcup$}}
\def\qed{\ifmmode\squareforqed\else{\unskip\nobreak\hfil
\penalty50\hskip1em\null\nobreak\hfil\squareforqed
\parfillskip=0pt\finalhyphendemerits=0\endgraf}\fi}

\def\cB{{\mathcal B}}

\def\cD{{\mathcal D}}

\def \Z {{\mathbb Z}}

 \def\\{\cr}
\def\({\left(}
\def\){\right)}

\def\mand{\qquad\mbox{and}\qquad}

 \numberwithin{dummy}{section}

\begin{document}

\title[Counting free tuples in $\SL_2(\mathbb{Z})$]{Counting embeddings of free groups into 
$\SL_2(\mathbb{Z})$ and its subgroups}

\author[K. Bulinski] {Kamil Bulinski}
\address{School of Mathematics and Statistics, University of New South Wales, Sydney NSW 2052, Australia}
\email{kamil.bulinski@gmail.com}

\author[A. Ostafe] {Alina Ostafe}
\address{School of Mathematics and Statistics, University of New South Wales, Sydney NSW 2052, Australia}
\email{alina.ostafe@unsw.edu.au}

\author[I. E. Shparlinski] {Igor E. Shparlinski}
\address{School of Mathematics and Statistics, University of New South Wales, Sydney NSW 2052, Australia}
\email{igor.shparlinski@unsw.edu.au}

\begin{abstract}  We show that if one selects uniformly independently and identically distributed matrices  $A_1, \ldots, A_s \in \SL_2(\Z)$ from a ball of large radius $X$ then with probability at least $1 - X^{-1 + o(1)}$ the matrices  $A_1, \ldots, A_s$ are free generators for a free subgroup of $\SL_2(\bZ)$.  Furthermore, to show the flexibility of our method we do similar counting for matrices from the congruence subgroup $\Gamma_0(Q)$ uniformly with respect to the positive integer 
$Q\le X$.
This improves and generalises a result of E.~Fuchs and I.~Rivin (2017) which claims that the probability is $1 + o(1)$. We also disprove one of the statements in their work that has been used to deduce their claim.
 \end{abstract}



\maketitle

\tableofcontents

\section{Introduction}
\subsection{Motivation and background} We are interested in counting $s$-tuples $(A_1, \ldots, A_s)$ of  $\SL_2(\mathbb{Z})$ matrices of bounded size such that $A_1, \ldots, A_s$ do not form a free generating set for a free subgroup of $\SL_2(\mathbb{Z})$. 

More precisely, 
for 
\begin{equation}
\label{eq:Matr A}
A =  \begin{bmatrix}
    a   & b \\
    c  & d \\
\end{bmatrix} \in\SL_2(\bZ)
\end{equation} 
we define its naive height as 
\[
\|A\| = \max\{|a|,|b|,|c|,|d|\}.
\] 
For $X>0$, we denote
\[
\SL_2(\Z;X)=\{A \in \SL_2(\Z) :~\|A\|\le X\}.
\]
For an integer $s\ge 1$ we define $N_s(X)$ as the number of non-free 
$s$-tuples 
\begin{equation}
\label{eq:s-tuple}
(A_1, \ldots, A_s) \in \SL_2(\Z;X)^s,
\end{equation}
that is, as the number of tuples~\eqref{eq:s-tuple}
 such that  for some $m \ge 1$,  $1 \le \nu_1 , \ldots, \nu_m \le s$ with $A_{\nu_i} \ne A_{\nu_{i+1}}^{-1}$, 
$i=1, \ldots, s-1$, we have 
\[
A_{\nu_1} \ldots  A_{\nu_m}= I_2,
\]
where, as usual,  $I_2 \in \SL_2(\bZ)$ denotes the identity matrix.

We recall that Fuchs and Rivin~\cite[Theorem~2.1]{FuRi} claim that for $s=2$ 
the group generated by a random pair $(A_1, A_2) \in \SL_2(\Z;X)^2$ 
 is free with 
probability $1 + o(1)$ as $X \to \infty$ and  in fact is  ``thin'' (that is, of 
infinite index in $\SL_2(\Z)$), which is equivalent to the bound
\[
N_2(X) = o( X^{4}).
\]  
Unfortunately the proof is based on~\cite[Lemma~2.5]{FuRi}, which seems to be false, see Section~\ref{sec:FR-wrong}.

Here we  count ``non-free'' $s$-tuples~\eqref{eq:s-tuple} for an arbitrary $s \ge 1$ and 
obtain a right bound on $N_s(X)$ with a power saving, see Corollary~\ref{cor:Nonfree SL2 Matr}.
In fact, to show the flexibility of our approach, we consider a more general question and count 
non-free $s$-tuples with matrices from the congruence subgroup 
\[
\Gamma_0(Q) = \left\{  \begin{bmatrix}
    a   & b \\
    c  & d \\
\end{bmatrix} \in\SL_2(\Z):~c \equiv 0 \pmod Q\right\}
\]
for an integer $Q\ge 1$. 
Furthermore, let 
$$
\Gamma_0(Q,X)= \Gamma_0(Q)\cap \SL_2(\Z;X).
$$
Then,  we define $N_s(Q,X)$ as the number of non-free 
$s$-tuples 
\begin{equation}
\label{eq:s-tuple G}
(A_1, \ldots, A_s)   \in \Gamma_0(Q,X)^s.
\end{equation}

\begin{thm}  
\label{thm:Nonfree Matr G}
For $X \to \infty$ we have 
\[
 N_s(Q,X)\le X^{2s -1+o(1)} Q^{-s+1}
\] 
\end{thm}  

It is useful to recall that by a classical result of Newman~\cite{New}, the total number of $s$-tuples~\eqref{eq:s-tuple} 
is of order $X^{2s}$.  However, it is not difficult to show that for $Q\le X$  we have 
\begin{equation}
\label{eq: Size G0}
\# \Gamma_0(Q, X) = X^{2+o(1)} Q^{-1}
\end{equation} 
(see our discussion in  Section~\ref{sec: asymptotic Gamma} and especially the lower bound~\eqref{eq:LB G0}). 
In particular the bound of  Theorem~\ref{thm:Nonfree Matr G} can be written as 
$$
 N_s(Q,X)\le  \(\# \Gamma_0(Q, X)\)^s (X/Q)^{-1+ o(1)} . 
$$

Using that $ N_s(X) =  N_s(1,X)$, we show that Theorem~\ref{thm:Nonfree Matr G} 
combined with a result of Eskin, Mozes and Shah give matching upper and lower
bounds on $N_s(X)$. 

\begin{cor}  
\label{cor:Nonfree SL2 Matr}
For $X \to \infty$ we have 
\[
X^{2s -1} \ll  N_s(X)\le X^{2s -1+o(1)}.
\]
\end{cor}

We also use this  as an opportunity  to give, in Section~\ref{sec:prelim},  a concise and self-contained review of the ping pong method considered by Fuchs-Rivin in \cite{FuRi} that gives a sufficient condition for a tuple of matrices in $\SL_2(\bZ)$ to be free. 


\subsection{Notation and conventions}  
We recall that  the notations $U = O(V)$, $U \ll V$ and $ V\gg U$  
are equivalent to $|U|\leqslant c V$ for some positive constant $c$, 
which throughout this work, are absolute.

We also write $f(X) = o(g(X))$ if for all $\varepsilon>0$ there exists $X_{\varepsilon}>0$ so that $|f(X)| \leq \varepsilon |g(X)|$ for all $X > X_{\varepsilon}$.

For an integer $k\ne 0$ we denote by $\tau(k)$ the number of integer positive divisors of $k$, 
for which we   use the well-known bound
\begin{equation}
\label{eq:tau}
\tau(k) = |k|^{o(1)} 
\end{equation}
as $|k| \to \infty$, see~\cite[Theorem~317]{HaWr}.

We always assume that the matrix $A$ is   as in~\eqref{eq:Matr A}. Since there are only $O(X)$ matrices  in  $\SL_2(\Z; X)$ 
with $c=0$ (namely, with $a=d= \pm 1$ and $|b|\le X$), we see from~\eqref{eq: Size G0}  that there are at most 
\[
X  \(X^{2+o(1)} Q^{-1}\)^{s-1} = X^{2s -1+o(1)} Q^{-s+1}
\]
$s$-tuples~\eqref{eq:s-tuple G} with at least one matrix $A_i$ having $c_i = 0$. 
Hence from now on  we simply discard these from our argument and always assume that 
\[
c \ne 0.
\] 
We also apply the same convention to indexed matrices $A_i$, which then have entries $a_i, b_i, c_i, d_i$.  
More formally we always work with matrices from the set 
$$
\Gamma_0^*(Q,X)=\{A\in \Gamma_0(Q,X):~ c\ne 0\}. 
$$

\section{Preliminaries on the ping pong method}
\label{sec:prelim}

\subsection{Actions of $\SL_2(\Z)$ matrices on the upper half-plane} 
 Let $$\bH = \{ x+yi \in \bC :~ x,y \in \bR, y > 0\}.$$ For $A\in \SL_2(\Z)$  and $z \in \bH$, we define the action of $A$ on $z$ in 
a standard way as  the {\it M{\"o}bius map\/} 
\begin{equation}
\label{eq:Mobius}
z \mapsto Az = \frac{az+b}{cz+d} \in \bH.
\end{equation}

We let 
\begin{align*}
&\Bsf(z_0,r) = \{ z \in \bH:~|z - z_0| < r \}, \\
&\Esf(z_0,r) = \{z \in \bH:~ |z-z_0|>r\}.
\end{align*}

We recall our convention  that throughout we assume that $A$ is a matrix of the form~\eqref{eq:Matr A} with $c \ne 0$
and similarly for matrices $A_1$ and $A_2$.

\begin{lemma} 
\label{lem:EtoB} For any $A\in \SL_2(\Z)$ the following hold. 
\begin{enumerate}
	\item[(i)] For any $z \in \Esf\left(\frac{-d}{c},\frac{1}{|c|}\right)$ we have that $Az \in \Bsf\left(\frac{a}{c}, \frac{1}{|c|}\right) $.
	\item[(ii)]  For any $z \in \Esf\left(\frac{a}{c}, \frac{1}{|c|} \right)$ we have that $A^{-1}z \in \Bsf\left(\frac{-d}{c}, \frac{1}{|c|}\right)$.
\end{enumerate}
\end{lemma}

\begin{proof} Assume $z \in \Esf\left(-d/c,1/|c|\right)$, thus $|cz+d| > 1$. Now 
\[\left | \frac{az+b}{cz+d} - \frac{a}{c} \right | = \frac{|bc-ad|}{|(cz+d)c|} = \frac{1}{|cz+d||c|} \leq \frac{1}{|c|}\]
 so indeed $Az \in \Bsf\left(\frac{a}{c}, \frac{1}{|c|}\right)$. 

The second claim follows from the first by swapping $A$ with 
 \[A^{-1} =  \begin{bmatrix}
    d   & -b \\
    -c  & a \\
\end{bmatrix}\] 
so we have the following swaps
\[\frac{-d}{c} \ \longleftrightarrow \ \frac{-a}{-c}  \mand  \frac{a}{c} \ \longleftrightarrow \ \frac{d}{-c},
\]
and the result follows.  \end{proof}

For $A$  as in~\eqref{eq:Matr A}  we define the open half-disk 
\[\cD (A) = \Bsf\left(\frac{a}{c}, \frac{1}{|c|}\right). \]
 Observe that 
 \[
 \cD\(A^{-1}\) = \Bsf\left(\frac{-d}{c}, \frac{1}{|c|}\right).
 \]
 
 We are also interested in the closures $\overline{\cD(A)}$ and $\overline{\cD\(A^{-1}\)}$. 

\begin{cor} 
\label{cor:Discs} 
If for $A \in \SL_2(\Z)$ we have $|\Tr(A)|>2$ then $\cD(A) \cap \cD\(A^{-1}\) = \emptyset$ and  for $z \notin \overline{\cD\(A^{-1}\)}$ we have $Az \in \cD(A)$.
\end{cor}

\begin{proof} The distance between $a/c$ and $-d/c$ is 
\[|a/c+d/c| = \frac{|\Tr(A)|}{c} > \frac{2}{c}
\] 
and so  $\cD(A)$ and $\cD\(A^{-1}\)$  are disjoint. Now if $z \notin \overline{\cD(A^{-1})}$ then 
$z \in \Esf\left(-d/c,1/|c|\right)$ and so $Az \in \cD(A)$ by   Lemma~\ref{lem:EtoB}.
\end{proof}

\subsection{Ping pong pairs} 

Corollary~\ref{cor:Discs} motivates the following definition.

\begin{defn} A pair of matrices $A_1,A_2 \in \operatorname{SL}_2(\bZ)$ is said to be a \textit{ping pong pair} if the four disks 
$\overline{\cD(A_1)}$, $\overline{\cD(A_1^{-1})}$, $ \overline{\cD(A_2)}$, $ \overline{\cD(A_2^{-1})}$ are all pairwise disjoint.
\end{defn}

We now have one of our main tools. 

\begin{lem} 
\label{lem:ping pong} If every pair $(A_i,A_j)$ in an $s$-tuple~\eqref{eq:s-tuple} 
 is a ping pong pair, then this $s$-tuple is free. 
\end{lem}

\begin{proof} Suppose that $B = B_m \ldots B_1$,  where $B_\nu\in \{A_1, A_1^{-1}, \ldots, A_s, A_s^{-1}\}$ for all $\nu=1, \ldots, m$ are such that $B_i \neq B_{i+1}^{-1}$ for all $i=1, \ldots, m-1$. Now, choose $z_0 \in \bH$ such that 
\[
z_0 \notin \bigcup_{i=1}^s\(\overline{\cD(A_i)} \cup \overline{\cD\(A_i^{-1}\)}\) .
\]
Thus $B_1z_0 \in \cD(B_1)$  by Corollary~\ref{cor:Discs}  above. Now, let 
\[
z_m = Bz_0 = B_m \ldots B_1 z_0.
\]
 We prove by induction on $m$ that $z_m \in \cD(B_m)$.
 
For $m=1$ this has just been established.

Now, suppose $z_m \in \cD(B_m)$.  Since $B_{m+1}^{-1} \neq B_m$ we have that $z_m \notin \cD(B_{m+1}^{-1})$. So we may use Corollary~\ref{cor:Discs}  to conclude that $z_{m+1} = B_{m+1}z_m \in \cD(B_{m+1})$, which completes the induction step. Thus $Bz_0 = z_m \neq z_0$ as $z_0 \notin \cD(B_m)$, which means that $B\ne I_2$. \end{proof}

\section{Counting ping pong $s$-tuples}

\subsection{Some counting results for  matrices from $ \Gamma_0^*(Q,X)$}

We need the following simple upper bound
\begin{lemma} \label{lemma: c_1 large} Let $Q\le Y\le X$. The number of matrices $A \in \Gamma_0^*(Q,X)$ such that $ |c| \le Y$ is at most $X^{1+o(1)}YQ^{-1}$.
\end{lemma}

\begin{proof}  
  If $bc+1 =0$ (which is possible only for $Q=1$) then we have at most two choices 
  for $(b,c) = \pm (1,-1)$ and since $ad=0$,  either $a= 0$ or $d=0$. Hence there are $O(X)$ 
  such matrices. 
  
 If $ad= 1 + bc\ne 0$, there are only at most $O(X)$ choices for $b$    and $O(YQ^{-1}) $ choices for $c$, thus $O(XYQ^{-1})$ choices for $bc+1$.  For each such $(b,c)$ pair, the number of choices for $(a,d)$ is $X^{o(1)}$ by~\eqref{eq:tau}. Altogether there are at most $X^{1+o(1)}YQ^{-1}$ such choices.
 \end{proof}
 
 We now count some special pairs of matrices $(A_1,A_2)$ for which $ \min\{ |c_1|, |c_2| \}$ falls in a dyadic interval. 

\begin{lemma}\label{lemma: estimate a_1/c_1} 
Let $Q\le Y\le X$. There are at most $X^{3+o(1)}Q^{-2} $ pairs of matrices $(A_1,A_2) \in \Gamma_0^*(Q,X)^2$ satisfying the following conditions:
\begin{enumerate}
	\item[(i)] $Y/2< \min\{ |c_1|, |c_2| \}\le Y$;
	\item[(ii)]  $\left| \frac{a_1}{c_1} - \frac{a_2}{c_2} \right| \leq \frac{1}{|c_1|} + \frac{1}{|c_2|}$.
\end{enumerate}
\end{lemma}

\begin{proof} Without loss of generality we may assume that $\min\{ |c_1|, |c_2| \}=|c_1|$.
 Hence by Lemma~\ref{lemma: c_1 large},  from Condition~(i) we have
$$
U_1 = X^{1+o(1)}YQ^{-1}
$$ choices for $A_1$. Assume now that $A_1$ is fixed.  

When $A_1$ is fixed, from Condition~(ii) we derive that 
\[
 \left| \alpha - \frac{a_2}{c_2} \right| \leq \Delta, 
 \]
 where  $\alpha = a_1/c_1$ and $\Delta = 4/Y$. Therefore, 
 \[
 \left| \alpha c_2  - a_2  \right| \leq \Delta |c_2|.
 \]
 Thus for a fixed $c_2$ there are $2\Delta |c_2| + 1$ choices for $a_2$.  For each such $a_2$, which we consider fixed,  
we conclude from  $a_2d_2 - b_2c_2 = 1$ that $a_2d_2\equiv 1 \pmod {c_2}$ and hence $d_2$ can take at most 
$2X/c_2 + 1 \ll X/c_2$ values after which $b_2$ is uniquely defined. Hence for a fixed $A_1$ there are at most 
\begin{align*}
U_2& \ll  \sum_{\substack{0<|c_2| \le X\\c_2 \equiv 0 \pmod Q}} (2\Delta c_2 + 1) X/c_2 \\
&=2\Delta X (X Q^{-1} +O(1))+ X Q^{-1}(\log (X/Q) + O(1))  \\ &\ll X^2Y^{-1}Q^{-1}  + XQ^{-1}\log X  
\end{align*}  
 choices for $A_2$. 
 
 Hence the  total number of pairs of matrices $(A_1,A_2) \in\SL_2(\Z;X)^2$ satisfying both Condition~(i) and~(ii)
 can be estimated as
 \[
 U_1 U_2 \le X^{1+o(1)}Y Q^{-1}  \( X^2Y^{-1}Q^{-1}  + X Q^{-1} \log X \) \le   X^{3+o(1)} Q^{-2}, 
 \]
 which concludes the proof.
\end{proof}

We can immediately deduce the following result which swaps $a_i$ with $d_i$ in the estimate.

\begin{lemma} 
\label{lemma: estimate -d_1/c_1}  
Let $Q\le Y\le X$. There are at most $X^{3+o(1)}Q^{-2} $ pairs of matrices $(A_1,A_2) \in \Gamma_0^*(Q,X)^2$ 
 satisfying the following conditions:
\begin{enumerate}
	\item $Y/2< \min\{ |c_1|, |c_2| \}\le Y$ 
	\item $\left| \frac{-d_1}{c_1} - \frac{-d_2}{c_2} \right| \leq \frac{1}{|c_1|} + \frac{1}{|c_2|}$.
\end{enumerate}
\end{lemma}

\begin{proof} The map $A \mapsto A^{-1}$ preserves the $\| \cdot \|$ norm and leaves $|c|$ unchanged, 
hence preserves the first two conditions. On the other hand, as
 \[
 A^{-1} = \begin{bmatrix}
    d   & -b \\
    -c  & a \\
\end{bmatrix}\] we see that the terms $\frac{a_i}{c_i}$ is swapped to $\frac{-d_i}{c_i}$ so the estimate follows from the previous lemma.
\end{proof}

\begin{lemma} 
\label{lemma: estimate a_2/c_2}  
Let $Q\le Y\le X$. There are at most $X^{3+o(1)}Q^{-2} $ pairs of matrices $(A_1,A_2) \in\Gamma_0^*(Q,X)^2$ satisfying the following conditions:
\begin{enumerate}
	\item $Y/2< \min\{ |c_1|, |c_2| \}\le Y$ 

	\item $\left| \frac{-d_1}{c_1} - \frac{a_2}{c_2} \right| \leq \frac{1}{|c_1|} + \frac{1}{|c_2|}$.
\end{enumerate}
\end{lemma}

\begin{proof} This follows from a similar argument, except this time we only apply the transformation $A_1 \mapsto A_1^{-1}$ to one of the matrices.
\end{proof}

\begin{lemma}\label{lem: trace count} For any fixed $t \in \bZ$ the number of matrices $A \in \Gamma_0^*(Q,X)$ with  $\Tr(A)=t$ is at most $X^{1+o(1)}$.
\end{lemma}

\begin{proof} As $a+d = t$ and $ad-bc=1$, we have $a(t-a) - bc = 1$. If $a(t-a) = 1$ there are are at most $2$ choices for $a$ and from 
$bc= 0$ we have $O(X)$ choices for $(b,c)$. 

If $a(t-a) \ne  1$, using the standard divisor bound~\eqref{eq:tau}, we see that for each choice of $a$ there are $X^{o(1)}$ choices for  $(b,c)$.

In each case, after $a$ is fixed, $d= t-a$ is uniquely defined.
 \end{proof}

\begin{lem}  \label{lem:Nonfree Pair}
The number of pairs of matrices $(A_1,A_2) \in\Gamma_0^*(Q,X)^2$ that do not form a ping pong pair is at most $X^{3+ o(1)} Q^{-1}$.
\end{lem}

\begin{proof} 
To ensure that $\cD(A_i)$ and $\cD(A_i^{-1})$ are disjoint, we just need $|\Tr(A_i)| >2$. 
There are only $X^{1+o(1)}$ such matrices $A_1$ or $A_2$ that fail this condition by 
Lemma~\ref{lem: trace count}. 

Thus there are at most $X^{3+o(1)} Q^{-1}$   pairs $(A_1,A_2) \in \Gamma_0^*(Q,X)^2$ with 
$$
\min\left\{|\Tr(A_1)|, |\Tr(A_2)|\right\} \le 2.
$$

We now count remaining pairs  $(A_1,A_2) \in \Gamma_0^*(Q,X)^2$. In particular $|c_1|, |c_2| \ge Q$. 

Let $Q\le Y\le X$.   
Now, $\cD(A_1)$ and $\cD(A_2)$ not being disjoint is equivalent to 
\begin{equation}
\label{eq:a1 c1 a2 c2}
\left| \frac{a_1}{c_1} - \frac{a_2}{c_2} \right| \leq \frac{1}{|c_1|} + \frac{1}{|c_2|} 
\end{equation}
appearing in Lemma~\ref{lemma: estimate a_1/c_1}. 

By Lemma~\ref{lemma: estimate a_1/c_1}, the number of pairs  $(A_1,A_2) \in \Gamma_0^*(Q,X)^2$ with $Y/2<\min\{|c_1|, |c_2|\} \le Y$ 
and satisfying~\eqref{eq:a1 c1 a2 c2} is at most $X^{3+o(1)} Q^{-2}$.
Likewise, we can estimate the number of pairs where $\cD(A_1)$ and $\cD(A_2^{-1})$ are not disjoint using the Lemmas~\ref{lemma: estimate -d_1/c_1} and~\ref{lemma: estimate a_2/c_2}, giving the same bound.

Taking $Y=X/{2^k}$, $k=0,\ldots,\lfloor\log (X/Q) \rfloor$, and using the fact that $O(\log X)=X^{o(1)}$, we conclude the proof.
\end{proof}

\subsection{Proof of Theorem~\ref{thm:Nonfree Matr G}}

For the upper bound, by Lemma~\ref{lem:Nonfree Pair} we see that the number of $s$-tuples~\eqref{eq:s-tuple G} 
which include a non-ping pong pair is at most
\[
X^{3+o(1)}Q^{-1}  \(X^2Q^{-1}\)^{s-2}= X^{2s-1+o(1)} Q^{-s+1}. 
\]
 Using  Lemma~\ref{lem:ping pong}  we conclude the proof.

\subsection{Proof of Corollary~\ref{cor:Nonfree SL2 Matr}}  

The upper bound follows from Theorem~\ref{thm:Nonfree Matr G} taken with $Q=1$. 

We now prove the lower bound. Let $\Phi_3(X)=X^2+X+1$ be the $3$rd cyclotomic polynomial. We note that any $2\times 2$ integer matrix $A$ with characteristic polynomial $\Phi_3$ must be in $\SL_2(\Z)$ since $\det A=\Phi(0)=1$. Therefore, by~\cite[Theorem~1.3]{EMS}, the number of matrices $A_1\in \SL_2(\Z;X)$ is at least $(C + o(1)) X$, with some absolute constant $C>0$. Therefore any such matrix $A_1\in\SL_2(\Z)$ satisfies $A_1^3=I_2$.

Choosing arbitrary 
 $A_2, \ldots, A_s$ we obtain
$$
N_s(X) \gg X^{1+2(s-1)}=X^{2s-1},
$$
which concludes the proof.

\section{Comments}  
\subsection{A lemma of Fuchs and Rivin~\cite{FuRi} }
\label{sec:FR-wrong}
We now turn to demonstrating that~\cite[Lemma~2.5]{FuRi} is wrong. For $A \in \SL_2(\Z)$ the norm
  \[\|A \|_{op} = \sqrt{\lambda_{\max} (A^t A)} \] was defined in \cite{FuRi} where $\lambda_{\max}$ denotes the largest eigenvalue. Note that $\| A\|_{op}$ is the largest singular value of $A$ and hence is the operator norm of $A$. For $X>0$ we let 
\[\cB_X = \{A \in \SL_2(\mathbb{Z}):~ \|A \|_{op} \leq X \}.\] 
It is well known that all matrix norms are equivalent and in particular 
\[
 \SL_2(\Z; \kappa_1X) \subseteq \cB_X\subseteq \SL_2(\Z; \kappa_2X)
\]
for some absolute constants $\kappa_1,\kappa_2>0$. 
So counting problems with respect to both norms are equivalent up to a constant.   

One easily checks that when $\Tr(A) \to \infty$ the fixed points of the M\"obius map~\eqref{eq:Mobius} approach $a/c$ and $-d/c$. 
It is  now evident that~\cite[Lemma~2.5]{FuRi} implies the following claim.
\begin{quotation}
{\it``Fix $r>0$. For $X>0$, if one samples uniformly $(A_1,A_2) \in \cB_X^2$ then \begin{equation} \label{eq: FR false bound} \left |\frac{a_1}{c_1} - \frac{a_2}{c_2} \right| > r \end{equation} holds with probability $1 + o(1)$ as $X \to \infty$.''}
\end{quotation}

To disprove this, we first note that for a random $A \in \cB_X$ the probability that $|c| \geq |a|$ is at least $1/2$ since the mapping 
\[
\begin{bmatrix}
    a   & b \\
    c  & d \\
\end{bmatrix} \mapsto \begin{bmatrix}
    0   & 1 \\
    -1  & 0 \\
\end{bmatrix} \begin{bmatrix}
    a   & b \\
    c  & d \\
\end{bmatrix} = \begin{bmatrix}
    c   & d \\
    -a  & -b \\
\end{bmatrix}
\] is a bijection on $\SL_2(\Z)$ that preserves the $\| \cdot \|_{op}$ norm (as it is a composition with an orthogonal matrix). Thus with probability at least $1/2$ we have that $-1 \leq  a/c \leq 1$. This already disproves the claim for $r>2$. For arbitrarily small $r$, say $r =1/k$ for some $k \in \bN$, the claim is still false since if one partitions $[-1,1]$ into $2k$ intervals of width $1/k$, then there is one such interval such that with probability at least 
\[
P_k \ge \frac{1}{2} \cdot \frac{1}{2k} =  \frac{1}{4k}
\] 
we have that $a/c$ lands in this interval. This means that with probability 
\[P_k^2 \ge \frac{1}{16k^2}=   \frac{1}{16} r^2
\] 
the inequality~\eqref{eq: FR false bound} fails.

We remark that our Lemma~\ref{lemma: estimate a_1/c_1} essentially replaces this fixed $r>0$ with $1/|c_1|+ 1/|c_2|$ 
and provides a quantitative estimate on the corresponding  probability.

\subsection{Lower bounds in higher dimensions}
We note that, using the same construction as in the proof of the lower bound of~\cite[Theorem 2.2]{OS}, one can also give a lower bound for the number of non-free $s$-tuples of matrices in $\SL_n(\Z;X)$ of any dimension $n\ge 2$. Indeed, for $n\ge 2$, we define
\[
w(n) = \max\left \{ \sum_{i=1}^h \varphi(k_i)^2:~n =\sum_{i=1}^h \varphi(k_i), ~ k_i \ge 2, \ i =1, \ldots, h\right\},
\]
where the maximum is taken over all such representations of all possible lengths $h \ge 1$ and, as usual, $\varphi(k)$ denotes the Euler function.  Then, as in the proof of~\cite[Theorem 2.2]{OS}, let $n = m_1+\ldots + m_h$ be a representation of $n$ as a sum of totients
$m_i = \varphi(k_i)$, $k_i \ge 2$,  $i =1,\ldots, h$.

Let $\Phi_{k_i}(X)$ be the $k_i$-th cyclotomic polynomial, $i =1,\ldots, h$, which is a monic irreducible polynomial of degree $m_i$ over $\Z$ and with the constant coefficient $\Phi_{k_i}(0)=1$. Therefore, any $m_i\times m_i$ integer matrix $B_i$ with characteristic polynomial $\Phi_{k_i}(X)$ satisfies $\det B_i=\Phi_{k_i}(0)=1$, and thus $B_i\in \SL_n(\Z)$. Moreover, $B_i^{k_i}=I_n$, $i=1,\ldots,h$, where $I_n \in \SL_n(\bZ)$ is the identity matrix

We now consider matrices $A_1\in \SL_n(\Z;X)$ of the form
\[
A_1 = \begin{pmatrix} B_1 &  &\mathbf{0}\\
\mathbf{0}&\ddots & \mathbf{0} \\
 \mathbf{0}&  & B_h
\end{pmatrix} 
\]
with diagonal cells formed by matrices $B_i \in  \SL_n \(\Z; X\)$ having characteristic polynomial $\Phi_{k_i}(X)$ and zeros everywhere else. 
Clearly for $k = k_1 \ldots  k_h$ we have $A_1^k=I_n$.  

Following now the computation in the proof of~\cite[Theorem 2.2]{OS}, using~\cite[Theorem~1.3]{EMS}, we see that there are 
\[
\# \{A_1\in \SL_n(\Z;X):~ A_1^n = I_n\}  \gg X^{-n/2} \prod_{i=1}^h  X^{m_i^2/2}.
\] 
Using now the asymptotic formula of~\cite{DRS}, we have, asymptotically, $(2X)^{(s-1)(n^2-n)}$ possibilities for arbitrary $(s-1)$-tuples $(A_2,\ldots,A_s)\in \SL_n(\Z;X)^{s-1}$. We thus conclude that the number $N_{s,n}(X)$ of non-free $s$-tuples $(A_1,\ldots,A_s)\in \SL_n(\Z;X)^s$ satisfies
\[
 N_{n,s}(X) \gg  X^{(s-1)n^2 -(s-1/2)n + w(n)/2} .
 \] 
 
\subsection{Asymptotic count of congruence subgroups}\label{sec: asymptotic Gamma}
 It is an interesting question to obtain an asymptotic formula for $\#\Gamma_0(Q,X)$.  
   However, for  the purpose of 
 understanding the strength of  Theorem~\ref{thm:Nonfree Matr G} the bound~\eqref{eq: Size G0} is quite enough.
 In fact, the upper bound in~\eqref{eq: Size G0} is instant from Lemma~\ref{lemma: c_1 large}
 taken with $Y = X$. 
 For the lower bound, we write $c = Qe$ for an positive integer $e \le X/Q$ and note that congruence 
 $ad \equiv 1 \pmod {c}$, $1 \le a,d \le  c$ has $\varphi(c)$ solutions. After this, we set 
 $b = (ad-1)/c$ and note that $0 \le b < c \le X$.  Thus, if $c \equiv 0 \pmod Q$ then the matrix $A$ defined by this 
 quadruple $(a,b,c,d)$ contributes to $\Gamma_0(Q,X)$.  So we write $c = Qe$ for an positive integer $e \le X/Q$ and 
using for a positive integer $k$ the Euler function satisfies
\[
  \varphi(k) \gg \frac{k}{\log \log (k+3)}, 
\]
see~\cite[Theorem~328]{HaWr},  we derive
 \begin{equation}
\label{eq:LB G0}
\begin{split}
\#  \Gamma_0(Q,X)& \ge \sum_{1 \le e \le X/Q} \varphi(eQ)
 \gg  \sum_{1 \le e \le X/Q} \frac{eQ}{\log \log ( eQ+3)}\\
 &   \ge  \frac{Q}{\log \log (X+3)}   \sum_{1 \le e \le X/Q}  e \gg  \frac{X^2}{Q \log \log (X+3)} .
\end{split}
\end{equation}

 Finally, we note that our approach also allows to get nontrivial bounds on the 
 number of non-free $s$ tuples of matrices from the other subgroups of $\SL_2(\Z)$ such as
 \[
\Gamma(Q) = \left\{  \begin{bmatrix}
   a   & b \\
   c  & d \\
\end{bmatrix} \in\SL_2(\Z):~ a,d \equiv 1 \pmod Q, \ b,c \equiv 0 \pmod Q\right\}
\]
and 
 \[
\Gamma_1(Q) = \left\{  \begin{bmatrix}
   a   & b \\
   c  & d \\
\end{bmatrix} \in\SL_2(\Z):~ a,d\equiv 1 \pmod Q, \ c  \equiv 0 \pmod Q\right\}.
\]
 


\section*{Acknowledgements}  
The authors are very grateful to Emmanuel Breuillard for indicating the paper~\cite{FuRi} and that their methods may 
lead to explicit bounds on our counting problem and to Julian Demeio for his interest and useful comments. 

During the preparation of this work, the authors were   supported in part by the  
Australian Research Council Grant  DP200100355.

\end{document}